\documentclass[runningheads]{llncs}

\usepackage[utf8]{inputenc}
\usepackage{graphicx}
\usepackage{subcaption}
\usepackage{amsmath,amssymb}
\usepackage{mathtools}
\usepackage{booktabs}
\usepackage[shortlabels]{enumitem}
\usepackage{color}
\usepackage[basic]{complexity}

\usepackage{footnote}
\makesavenoteenv{tabular}

\usepackage[nocompress,noadjust]{cite}
\usepackage[dvipsnames,usenames,cmyk]{xcolor}
\usepackage[colorlinks=true, urlcolor=blue, citecolor=green, linkcolor=blue, breaklinks, unicode]{hyperref}
\usepackage[nameinlink,capitalize]{cleveref}

\newtheorem{prop}{Proposition}

\newenvironment{proofof}[1]{\medskip\noindent\textit{Proof of #1.}}{\mbox{}\hfill\qed\par\medskip}

\graphicspath{{pics/}}

\newcommand{\df}[1]{{\it #1}}

\newcommand{\CC}{{\ensuremath{\mathcal{C}}}}
\DeclareMathOperator{\qn}{qn}

\newcommand{\SCon}{{\ensuremath{~\cup~}}}

\renewcommand{\orcidID}[1]{\href{https://orcid.org/#1}{\includegraphics[scale=.03]{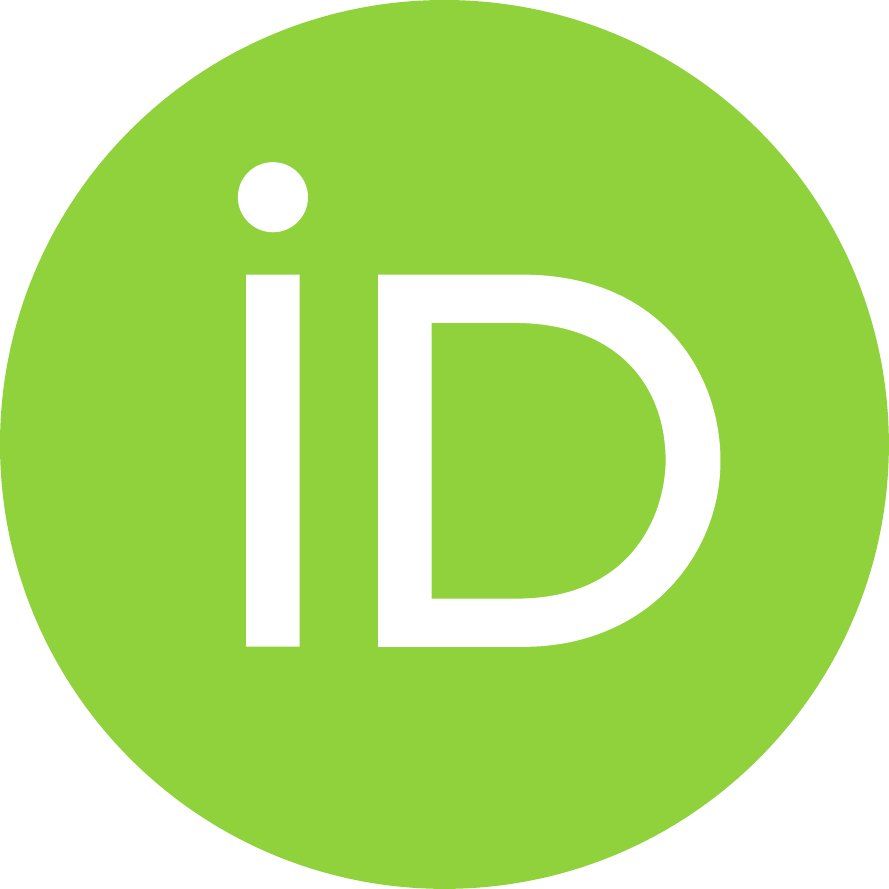}}}

\begin{document}

\title{Queue Layouts of Two-Dimensional Posets}

\author{
	Sergey~Pupyrev\orcidID{0000-0003-4089-673X}
}

\authorrunning{S. Pupyrev}

\institute{
	\email{spupyrev@gmail.com}
}

\date{}
\maketitle

\begin{abstract}
	The queue number of a poset is the queue number of its cover graph when the vertex order is
	a linear extension of the poset. Heath and Pemmaraju conjectured that 
	every poset of width $w$  has queue number at most $w$. The conjecture has been
	confirmed for posets of width $w=2$ and for planar posets with $0$ and $1$.
	In contrast, the conjecture has been refused by a family of general (non-planar) posets
	of width $w>2$.\newline 
	In this paper, we study queue layouts of two-dimensional posets. First, we construct a 
	two-dimensional poset of width $w > 2$ with queue number $2(w - 1)$, thereby
	disproving the conjecture for two-dimensional posets. Second, we show an upper
	bound of $w(w+1)/2$ on the queue number of such posets, thus improving the previously 
	best-known bound of \linebreak
	$(w-1)^2+1$ for every $w > 3$.
	
\end{abstract}

\section{Introduction}

Let $G$ be a simple, undirected, finite graph with vertex set $V$ and edge set $E$, and 
let $\sigma$ be a total order of $V$. 
For a pair of distinct vertices $u$ and $v$, we write $u <_{\sigma} v$ (or simply $u < v$), if $u$ precedes
$v$ in $\sigma$. We also write $[v_1, v_2, \dots, v_k]$ to denote that $v_i$ precedes $v_{i+1}$
for all $1 \le i < k$; such a subsequence of $\sigma$ is called a \df{pattern}.
Two edges $(u, v) \in E$ and $(a, b) \in E$ \df{nest} if $u <_{\sigma} a <_{\sigma} b <_{\sigma} v$. 
A $k$-queue layout of $G$ is a total order of $V$ and a partition of $E$ into subsets $E_1, E_2, \dots, E_k$,
called \df{queues}, such that no two edges in the same set $E_i$ nest. The \df{queue number}
of $G$, $\qn(G)$, is the minimum $k$ such that $G$ admits a $k$-queue layout.
Equivalently, the queue number is the minimum $k$ such that there exists an order $\sigma$
containing no \df{$(k+1)$-rainbow}, that is, a set of edges $\{(u_i, v_i); i=1,2,\dots, k+1\}$
forming pattern $[u_1, \dots, u_{k+1}, v_{k+1}, \dots, v_1]$ in $\sigma$.

Queue layouts can be studied for partially ordered sets (or simply \df{posets}).
A poset over a finite set of elements $X$ is a transitive and asymmetric binary relation $<$ on $X$.
The main idea is that given a poset, one should lay it out respecting the relation.
Two elements $a, b$ of a poset, $P = (X, <)$, are called \df{comparable} if $a < b$ or $b < a$, 
and \df{incomparable}, denoted by $a \parallel b$, otherwise.
Posets are visualized by their diagrams: Elements are placed as points
in the plane and whenever $a < b$ in the poset and there is no element $c$ with
$a < c < b$, there is a curve from $a$ to $b$ going upwards (that is $y$-monotone); see \cref{fig:introA}.
Such relations, denoted by $a \prec b$, are known as \df{cover relations}; they are 
essential in the sense that they are not implied by transitivity. 
The directed graph implicitly defined by such a diagram is the cover
graph $G_P$ of the poset $P$. Given a poset $P$, a linear extension $L$ of $P$ is a total order on
the elements of $P$ such that $a <_L b$, whenever $a <_P b$. Finally, the queue number of a 
poset $P$, denoted by $\qn(P)$, is the smallest
$k$ such that there exists a linear extension $L$ of $P$ for which the resulting layout
of $G_P$ contains no $(k + 1)$-rainbow; see \cref{fig:introC}.

Queue layouts of posets were first studied by Heath and Pemmaraju~\cite{HP97}, who provided bounds
on the queue number of posets in terms of their \df{width}, that is, the maximum number of
pairwise incomparable elements. In particular, they observed that the size of a rainbow
in a queue layout of a poset of width $w$ cannot exceed $w^2$, and therefore, $\qn(P) \le w^2$
for every poset $P$. Furthermore, Heath and Pemmaraju conjectured that $\qn(P) \le w$ for a 
width-$w$ poset $P$. The study of the conjecture received a notable attention in the recent years.
Knauer, Micek, and Ueckerdt~\cite{KPT18} confirmed the conjecture for posets of width $w=2$ and 
for planar posets with $0$ and $1$. Later Alam et al.~\cite{Alam20} constructed a
poset of width $w \ge 3$ whose queue number is $w+1$, thus refuting the conjecture for general non-planar posets. 
In the same paper Alam et al. improved the upper bound by showing that $\qn(P) \le (w-1)^2+1$
for all posets $P$ of width $w$. Finally, Felsner, Ueckerdt, and Wille~\cite{FUW21}
strengthened the lower bound by presenting a poset of width $w>3$ with $\qn(P) \ge w^2 / 8$.

In this short paper we refine our knowledge on queue layouts of posets by improving the
known upper and lower bounds of the queue number of two-dimensional posets.
Recall that the \df{dimension} of poset $P$ is the least positive integer $d$ for which there 
are $d$ linear extensions (\df{realizers}) $L_1, \dots, L_d$ of $P$ so
that $a < b$ in $P$ if and only if $a < b$ in $L_i$ for every $i \in \{1, \dots, d\}$.
\df{Two-dimensional} posets are described by realizers $L_1$ and $L_2$ and 
often represented by \df{dominance drawings} in which the coordinates of the elements are
their positions in $L_1$ and $L_2$; see \cref{fig:introB}.
We emphasize that the existing lower bound constructions~\cite{Alam20,FUW21} are not 
two-dimensional. Thus, Felsner et al.~\cite{FUW21} asked whether the
conjecture of Heath and Pemmaraju holds for posets with dimension $2$. Our first result
answers the question negatively.

\begin{theorem}
	\label{thm:lb}
	There exists a two-dimensional poset $P$ of width $w>1$ with $\qn(P) \ge 2(w-1)$.
\end{theorem}	

Observe that our construction and the proof of \cref{thm:lb} for $w=3$ is arguably much simpler
than the one of Alam et al.~\cite{Alam20}, which is based on a tedious case analysis. Thus, it can
be interesting on its own right.

Next we study the upper bound on the queue number of two-dimensional posets. Our result is the
following theorem, which is an improvement over the known $(w-1)^2+1$ bound of Alam et al.~\cite{Alam20}
for every $w>3$.

\begin{theorem}
	\label{thm:ub}
	Let $P$ be a two-dimensional poset with realizers $L_1, L_2$. 
	Then there is a layout of $P$ in at most $w(w+1)/2$ queues using either $L_1$ or $L_2$ as 
	the vertex order.
\end{theorem}	

The paper is structured as follows. In \cref{sect:lb} we prove \cref{thm:lb} and 
in \cref{sect:ub} we prove \cref{thm:ub}. \cref{sect:conc} concludes the paper with interesting open questions. 

\begin{figure}[!ht]
	\begin{subfigure}[b]{.3\linewidth}
		\center
		\includegraphics[page=1]{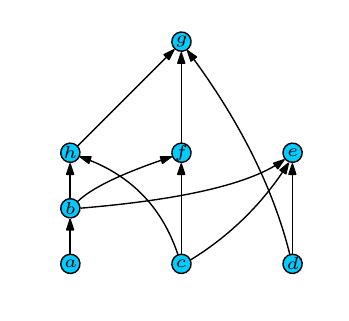}
		\caption{}
		\label{fig:introA}
	\end{subfigure}
	\begin{subfigure}[b]{.3\linewidth}
		\center
		\includegraphics[page=2]{pics/intro}
		\caption{}
		\label{fig:introB}
	\end{subfigure}    
	\begin{subfigure}[b]{.38\linewidth}
		\center
		\includegraphics[page=3]{pics/intro}
		\caption{}
		\label{fig:introC}
	\end{subfigure}
	\caption{A two-dimensional poset of width $3$, its dominance drawing, and a $2$-queue layout}
	\label{fig:intro}
\end{figure}

\section{An Upper Bound}
\label{sect:ub}

Consider a two-dimensional poset $P=(X, <)$ of width $w \ge 1$ with realizers $L_1$ and $L_2$. 
In this section we study queue layouts of $P$ using vertex orders $L_1$ or $L_2$, which we call \df{realizer-based}.
It is well-known that the elements of $P$ can be partitioned into $w$ \df{chains}, that is, subsets of pairwise comparable
elements. We fix such a partition and treat it as a function $\CC : X \rightarrow \{1, \dots, w\}$ 
such that if $\CC(u) = \CC(v)$ and $u \neq v$, then either $u < v$ or $v < u$.

We start with a property of a linear extension of a poset, whose proof follows directly
from the absence of transitive edges in $G_P$. Recall that $\prec$ indicates cover relations of $P$, 
that is, edges of $G_P$.

\begin{prop}
	\label{prop:bbb}
	A linear extension of a poset $P$ with chain partition $\CC$ does not contain pattern
	$[b_1, b_2, b_3]$, 
	where $\CC(b_1) = \CC(b_2) = \CC(b_3)$ and $b_1 \prec b_3$.
\end{prop}

The next observation, whose proof is immediate, provides a crucial property of realizer-based linear 
extensions of two-dimensional posets.
In fact, a poset, $P$, admits a linear extension with such a property if and only if $P$ has dimension $2$;
see for example~\cite{DM41} where such linear extensions are called \df{non-separating}.

\begin{prop}
	\label{prop:abc}
	Consider a two-dimensional poset $P$ with realizers $L_1, L_2$ and chain partition $\CC$.
	Let $[a_1, b, a_2]$ be a pattern in $L_1$ (or $L_2$) with $\CC(a_1) = \CC(a_2)$.
	Then either $a_1 < b$ or $b < a_2$.
\end{prop}

The next useful property in the section holds for realizer-based linear 
extensions of two-dimensional posets.

\begin{prop}
	\label{prop:abaab}
	Consider a two-dimensional poset $P$ with realizers $L_1, L_2$ and chain partition $\CC$.
	Then $L_1$ (or $L_2$) does not contain pattern $[a_1, b_2, a, a_2, b_1]$, where
	$\CC(a_1) = \CC(a_2) = \CC(a)$, $\CC(b_1) = \CC(b_2)$, and $a_1 \prec b_1$, $b_2 \prec a_2$.
\end{prop}

\begin{proof}
	For the sake of contradiction, assume that $[a_1, b_2, a, a_2, b_1]$ is in $L_1$, with
	$\CC(a_1) = \CC(a_2) = \CC(a)$, $\CC(b_1) = \CC(b_2)$, and $a_1 \prec b_1$, $b_2 \prec a_2$.
	Notice that $a_1 \parallel b_2$, as otherwise we have $a_1 < b_2 < b_1$ and the edge $(a_1, b_1)$
	is transitive. Hence by \cref{prop:abc} applied to $[a_1, b_2, a]$, $b_2 < a$.
	Therefore, it holds that $b_2 < a < a_2$, which contradicts to non-transitivity of edge $(b_2, a_2)$.
\end{proof}	

Now we ready to prove the main result of the section.

\begin{proofof}{\cref{thm:ub}}
	Assume that poset $P$ is partitioned into $w$ chains, and consider a maximal rainbow, denoted $T$,
	induced by the order $L_1$. We need to prove that $|T| \le w (w+1) / 2$.
	
	First observe that the rainbow, $T$, does not contain two distinct edges $(a_1, b_1)$ and $(a_2, b_2)$
	with $\CC(a_1) = \CC(a_2)$ and $\CC(b_1) = \CC(b_2)$. Otherwise, the former edge nests the latter one and
	we have $a_1 < a_2 \prec b_2 < b_1$, which violates non-transitivity of $(a_1, b_1)$.
	Therefore, we already have $|T| \le w^2$. (This is the argument of Heath and Pemmaraju for their original 
	upper bound in~\cite{HP97})
	
	Next we show two more configurations that are absent in $T$:	
	\begin{enumerate}[(i)]
		\item For every pair of distinct chains, the rainbow does not contain edges 
		$(a_1, b_1)$, $(b_2, a_2)$, and $(a_3, a_4)$ with
		$\CC(a_1) = \CC(a_2) = \CC(a_3) = \CC(a_4)$ and $\CC(b_1) = \CC(b_2)$.
				
		For a contradiction, assume the rainbow contains the three edges. By \cref{prop:bbb},
		edge $(a_3, a_4)$ cannot cover elements $a_1$ or $a_2$. Thus, 		
		$L_1$ contains pattern $[a_1, b_2, a_3, a_4, a_2, b_1]$ or $[b_2, a_1, a_3, a_4, b_1, a_2]$.
		Both patterns violate \cref{prop:abaab}.
		
		\item For every triple of distinct chains, the rainbow does not contain edges
		$(a_1, b_1)$, $(b_2, a_2)$, $(a_3, c_3)$, and $(c_4, a_4)$ with 
		$\CC(a_1) = \CC(a_2) = \CC(a_3) = \CC(a_4)$, $\CC(b_1) = \CC(b_2)$, and $\CC(c_3) = \CC(c_4)$.
		
		For a contradiction, assume $T$ contains the four edges. Consider the inner-most edge in the rainbow; 
		without loss of generality, assume the edge is $(a_1, b_1)$. 
		Vertex $a_1$ is covered by two edges, $(a_3, c_3)$ and $(c_4, a_4)$, forming the 
		pattern of \cref{prop:abaab}; a contradiction.
	\end{enumerate}	

	Now observe that $T$ may contain at most $w$ \df{uni-colored} edges (that is, $(u, v)$ such that $\CC(u) = \CC(v)$)
	and at most $w(w-1)$ \df{bi-colored} edges (that is, $(u, v)$ such that $\CC(u) \neq \CC(v)$).
	
	On the one hand, if $T$ contains exactly $w$ uni-colored edges and $|T| > w (w+1) / 2$, then it must contain
	at least one pair of bi-colored edges $(a_1, b_1)$, $(b_2, a_2)$ with $\CC(a_1) = \CC(a_2)$, $\CC(b_1) = \CC(b_2)$.
	Together with the uni-colored edge from chain $\CC(a_1)$, the triple forms the forbidden configuration~$(i)$.
	
	On the other hand, if $T$ contains at most $w-1$ uni-colored edges and $|T| > w (w+1) / 2$, then
	$T$ contains two pairs of bi-colored edges, as in configuration~$(ii)$; a contradiction.
	
	This completes the proof of the theorem.
\end{proofof}
	
Notice that the bound of \cref{thm:ub} is worst-case optimal, as we show next.

\begin{lemma}
	\label{lm:worst_case}
	There exists a two-dimensional poset of width $w \ge 1$, denoted $R_w$, with realizers $L_1, L_2$
	such that its layout with vertex order $L_1$ contains a $\big(w(w+1)/2\big)$-rainbow.
\end{lemma}
	
\begin{proof}
	The poset $R_w$ is built recursively. For $w=1$, the poset consists of two comparable elements.
	For $w > 1$, we assume that $R_{w-1}$ is constructed and described by realizers $L_1^{w-1}$ and 
	$L_2^{w-1}$. The poset $R_w$ is constructed from $R_{w-1}$ by adding $2w$ elements. 
	Assume $|L_1^{w-1}| = n$ and the elements of $R_{w-1}$ are indexed by $w+1, \dots, w+n$. 
	We set $L_1^w$ to the identity permutation and use
	\[
	L_2^w = L_2^{w-1} \SCon (1, n+2w, 2, n+2w-1, \dots, w, n+w+1),
	\]
	where $\cup$ denotes the concatenation of the two orders. \cref{fig:rw} illustrates the construction.
	It is easy to verify that the width of the new poset is exactly $w$.
	Observe that in the layout of $R_w$ with order $L_1^w$, edges $(1, n+2w), \dots, (w, n+w+1)$ 
	form a $w$-rainbow and nest all edges of $R_{w-1}$. Therefore, the layout contains a $\big(w(w+1)/2\big)$-rainbow, as claimed.
\end{proof}

\begin{figure}[!tb]
	\center
	\includegraphics[page=4]{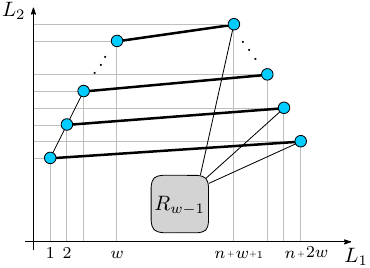}
	\caption{A 2-dimensional poset of width $w\ge 1$, $R_w$, with a realizer-based order 
		containing a $\big(w(w+1)/2\big)$-rainbow, which is comprised of $w$ thick edges that nest all edges of $R_{w-1}$}
	\label{fig:rw}
\end{figure}

We remark that \cref{lm:worst_case} provides a poset whose queue layout with one of its realizers
contains a $\big(w(w+1)/2\big)$-rainbow. It is straightforward to extend the construction
(by concatenating $R_w$ with its dual) so that both realizer-based vertex orders yield a rainbow of that size. 
However, the queue number of the poset (and the proposed extension) is at most $w$, which is achieved 
with a different, non-realizer-based, vertex order. Thus, a more delicate construction is needed to force
a larger rainbow in every linear extension of a poset.

\section{A Lower Bound}
\label{sect:lb}

In this section we provide a new counter-example to the conjecture of Heath and Pemmaraju~\cite{HP97}
by describing a two-dimensional poset of width $w \ge 3$ whose queue number exceeds $w$.
The poset, denoted $P_w$, is constructed recursively. The base case, $P_2$, is a four-element poset
with $L_1 = (1, 2, 3, 4)$ and $L_2 = (2, 1, 4, 3)$; see \cref{fig:lb_p2}.
The step of the construction is illustrated in \cref{fig:lb_pw}. 
Poset $P_w$ consists of a copy of $P_{w-1}$, a copy of the poset $R_{w-1}$ utilized
in \cref{lm:worst_case}, the duals of the two posets, and a chain of additional elements. 
Recall that the \df{dual} of a poset, $P$, is the poset, $\overline{P}$, on the same set of elements such 
that $x < y$ in $P$ if and only if $y < x$ in $\overline{P}$ for every pair of the elements $x$ and $y$.

We now formally describe the construction. Denote by $L_1(P), L_2(P)$ the two realizers of a two-dimensional poset $P$.
Let $\cup$ denote the concatenation of two sequences, and
let $(x_1, x_2, \dots) \uplus (y_1, y_2, \dots)$ denote the \df{interleaving} of two equal-length
sequences, that is, $(x_1, y_1, x_2, y_2, \dots)$.
Assume that $R_{w-1}$ contains $r$ elements. Then we set
\begin{align*}
	L_1(P_w) =~~ & (x_1, \dots, x_{r+1}) \SCon b \SCon s \SCon y_1 \SCon \big(L_1(R_{w-1}) \uplus (y_2, \dots, y_{r+1}) \big) \SCon \\
	& L_1(P_{w-1}) \SCon a \SCon L_1(\overline{P_{w-1}}) \SCon L_1(\overline{R_{w-1}}) \SCon t, \text{~and~}\\
	L_2(P_w) =~~ & s \SCon L_2(R_{w-1}) \SCon L_2(P_{w-1}) \SCon a \SCon L_2(\overline{P_{w-1}}) \SCon \\
	& \big( (x_1, \dots, x_r) \uplus L_2(\overline{R_{w-1}}) \big) \SCon x_{r+1} \SCon t \SCon b \SCon (y_1, \dots, y_{r+1}).
\end{align*}

\begin{figure}[!tb]
	\begin{minipage}[c][][t]{.4\textwidth}
		\centering
		\includegraphics[page=2]{lower_bound}
		\subcaption{$R_2$}
		\label{fig:lb_r2}\par\vfill
		\includegraphics[page=3]{lower_bound}
		\subcaption{$P_2$}
		\label{fig:lb_p2}
	\end{minipage}
	\begin{minipage}[c][][t]{.6\textwidth}
		\centering
		\includegraphics[page=1]{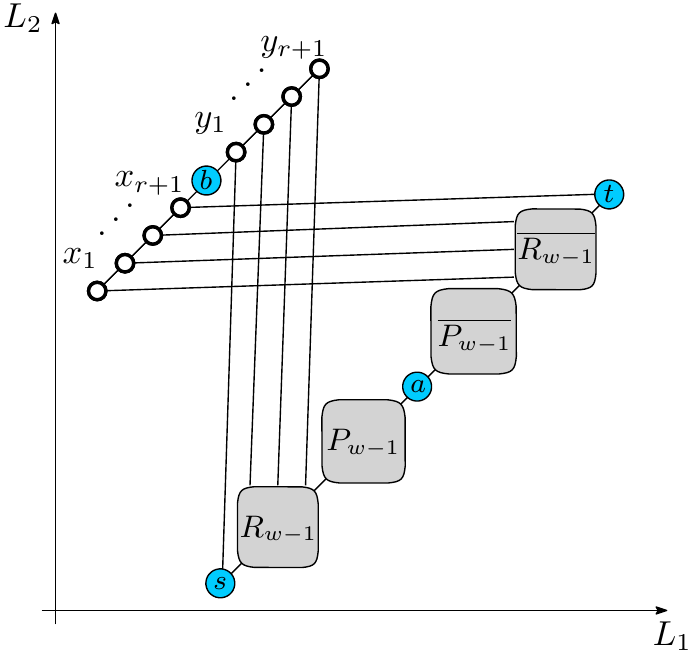}
		\subcaption{$P_w$}
		\label{fig:lb_pw}
	\end{minipage}%
	\caption{A counter-example to the conjecture of Heath and Pemmaraju~\cite{HP97}: A~two-dimensional poset, $P_w$, 
		of width $w \ge 3$ with queue number exceeding $w$}
	\label{fig:lb}
\end{figure}

We refer to \cref{fig:lb} for the illustration of the construction and to \cref{fig:lb3} for the instance of $P_3$.
Now we prove that the constructed poset has queue number at least $2w-2$.


\begin{proofof}{\cref{thm:lb}}
	It is easy to verify that the constructed poset, $P_w$, is two-dimensional and has width exactly 
	$w$. Furthermore, the poset is dual to itself, that is, $P_w = \overline{P_w}$ with $a$ and $b$ 
	being the fixed points. Thus, we may assume that in the linear extension corresponding to the
	optimal queue layout of the poset, element $a$ precedes $b$ and we have $s < \dots < a < b < y_1 < \dots < y_{r+1}$. 
	Next we consider the queue layout induced
	by the elements $s$, $R_{w-1}$, $P_{w-1}$, $a$, and $y_1, \dots, y_{r+1}$; see \cref{fig:lb_lin}.
	
	We prove the theorem by induction. For $w=2$, the claim holds trivially. For $w > 2$, we assume
	that $\qn(P_{w-1}) \ge 2(w-2)$ and distinguish two cases depending on the size of the maximum
	rainbow, $T$, formed by edges $(s, y_1)$, $(v_1, y_2), \dots, (v_r, y_{r+1})$, where
	$v_i, 1 \le i \le r$ are elements of $R_{w-1}$:
	\begin{itemize}
		\item if $|T| \ge 2$, then $\qn(P_w) \ge \qn(P_{w-1}) + |T| \ge 2(w-1)$, as all edges of $P_{w-1}$
		are nested by edges of $T$;
		
		\item if $|T| = 1$, then the elements of $R_{w-1}$ must appear in the order induced by
		$L_1(R_{w-1})$, since otherwise at least two of the edges of $T$ nest. By \cref{lm:worst_case}, 
		the edges of $R_{w-1}$ form a $\big(w(w-1)/2\big)$-rainbow. The rainbow is covered by edge $(s, y_1)$, 
		which yields $\qn(P_w) \ge \big(w(w-1)/2\big) + 1 \ge 2(w-1)$ for $w \ge 3$.
	\end{itemize}

	This completes the proof of \cref{thm:lb}.
\end{proofof}
	
\begin{figure}[!tb]
	\center
	\includegraphics[page=6]{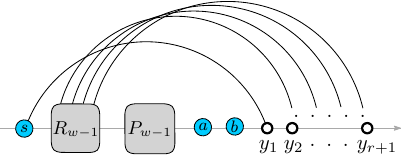}
	\caption{A linear extension of poset $P_w$ for the proof of \cref{thm:lb} in which $a < b$}
	\label{fig:lb_lin}
\end{figure}
	
\section{Conclusions}
\label{sect:conc}
We disproved the conjecture of Heath and Pemmaraju for two-dimensional posets and answered a
question posed by Felsner et al.~\cite{FUW21}. A number of intriguing problems in the area 
remain unsolved. 

\begin{itemize}
	\item Is it possible to get a subquadratic upper bound on the queue number of two-dimensional
	posets of width $w$? A poset of Felsner et al.~\cite{FUW21} that requires $w^2 / 8$ queues
	in every linear extension is not two-dimensional, which leaves a hope for an asymptotically
	stronger result than the one given by \cref{thm:ub}.
	
	\item What is the queue number of two-dimensional posets of width $3$? By \cref{thm:lb} and the
	result of Alam et al.~\cite{Alam20}, the value is either $4$ or $5$.
	
	\item Queue layouts of graphs are closely related to so-called \df{track layouts}, which
	are connected with the existence of low-volume three-dimensional graph drawings~\cite{DPW04,Pup20}.
	In particular, every $t$-track (undirected) graph has a $(t-1)$-queue layout, and every $q$-queue 
	(undirected) graph has track number at most $4q \cdot 4q^{(2q-1)(4q-1)}$. 
	We think it is interesting to study the relationship between the two concepts for directed graphs
	and posets.
\end{itemize}

\paragraph*{Acknowledgments.}	
We thank Jawaherul Alam, Michalis Bekos, Martin Gronemann, and Michael Kaufmann for 
fruitful initial discussions of the problem.

\bibliographystyle{splncs04}
\bibliography{refs}

\begin{thebibliography}{1}
\providecommand{\url}[1]{\texttt{#1}}
\providecommand{\urlprefix}{URL }
\providecommand{\doi}[1]{https://doi.org/#1}

\bibitem{Alam20}
Alam, J.M., Bekos, M.A., Gronemann, M., Kaufmann, M., Pupyrev, S.: Lazy queue
  layouts of posets. In: Auber, D., Valtr, P. (eds.) Graph Drawing and Network
  Visualization. pp. 55--68. Springer International Publishing (2020).
  \doi{10.1007/978-3-030-68766-3\_5}

\bibitem{DPW04}
Dujmovi{\'c}, V., P{\'{o}}r, A., Wood, D.R.: Track layouts of graphs. Discrete
  Mathematics {\&} Theoretical Computer Science  \textbf{6}(2),  497--522
  (2004). \doi{10.46298/dmtcs.315}

\bibitem{DM41}
Dushnik, B., Miller, E.W.: Partially ordered sets. American Journal of
  Mathematics  \textbf{63}(3),  600--610 (1941). \doi{10.2307/2371374}

\bibitem{FUW21}
Felsner, S., Ueckerdt, T., Wille, K.: On the queue-number of partial orders.
  In: Purchase, H.C., Rutter, I. (eds.) Graph Drawing and Network
  Visualization. pp. 231--241. Springer International Publishing (2021).
  \doi{10.1007/978-3-030-92931-2\_17}

\bibitem{HP97}
Heath, L.S., Pemmaraju, S.V.: Stack and queue layouts of posets. SIAM Journal
  on Discrete Mathematics  \textbf{10}(4),  599--625 (1997).
  \doi{10.1137/S0895480193252380}

\bibitem{KPT18}
Knauer, K., Micek, P., Ueckerdt, T.: The queue-number of posets of bounded
  width or height. In: Biedl, T., Kerren, A. (eds.) Graph Drawing and Network
  Visualization. pp. 200--212. Springer International Publishing (2018).
  \doi{10.1007/978-3-030-04414-5\_14}

\bibitem{Pup20}
Pupyrev, S.: Improved bounds for track numbers of planar graphs. Journal of
  Graph Algorithms and Applications  \textbf{24}(3),  323--341 (2020).
  \doi{10.7155/jgaa.00536}

\end{thebibliography}

\newpage
\appendix

\section*{Additional Illustrations}

\begin{figure}[!h]
	\center
	\includegraphics[page=5]{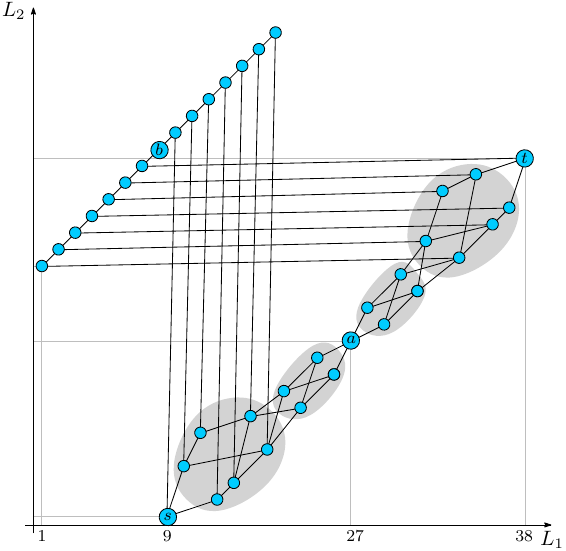}
	\caption{A two-dimensional poset with $38$ elements and width $3$. The queue number of the 
		poset is exactly $4$; the lower bound is shown in \cref{thm:ub}, and the upper bound
	is verified computationally via an open source SAT-based solver available at \url{http://be.cs.arizona.edu}}
	\label{fig:lb3}
\end{figure}

\end{document}